\documentclass{amsart}
\input{psfig.sty}
\input{epsf.tex}
\usepackage{amsfonts,amssymb,amsmath,latexsym,amscd}
\usepackage{epsfig,verbatim}
\usepackage{amsthm,graphics,textcomp}
\usepackage{graphicx}
\usepackage{color}
\usepackage{url}
\usepackage{epsfig}
\input xy
\xyoption{all}

\begin{document}

\newtheorem{theorem}{Theorem}[section]
\newtheorem{prop}[theorem]{Proposition}
\newtheorem{lemma}[theorem]{Lemma}
\newtheorem{cor}[theorem]{Corollary}

\theoremstyle{definition}
\newtheorem{definition}[theorem]{Definition}

\newcommand{\C}{{\mathbb C}}
\newcommand{\integers}{{\mathbb Z}}
\newcommand{\natls}{{\mathbb N}}
\newcommand{\ratls}{{\mathbb Q}}
\newcommand{\R}{{\mathbb R}}
\newcommand{\proj}{{\mathbb P}}
\newcommand{\lhp}{{\mathbb L}}
\newcommand{\tube}{{\mathbb T}}
\newcommand{\cusp}{{\mathbb P}}
\newcommand\AAA{{\mathcal A}}
\newcommand\BB{{\mathcal B}}
\newcommand\CC{{\mathcal C}}
\newcommand\DD{{\mathcal D}}
\newcommand\EE{{\mathcal E}}
\newcommand\FF{{\mathcal F}}
\newcommand\GG{{\mathcal G}}
\newcommand\HH{{\mathcal H}}
\newcommand\II{{\mathcal I}}
\newcommand\JJ{{\mathcal J}}
\newcommand\KK{{\mathcal K}}
\newcommand\LL{{\mathcal L}}
\newcommand\MM{{\mathcal M}}
\newcommand\NN{{\mathcal N}}
\newcommand\OO{{\mathcal O}}
\newcommand\PP{{\mathcal P}}
\newcommand\QQ{{\mathcal Q}}
\newcommand\RR{{\mathcal R}}
\newcommand\SSS{{\mathcal S}}
\newcommand\TT{{\mathcal T}}
\newcommand\UU{{\mathcal U}}
\newcommand\VV{{\mathcal V}}
\newcommand\WW{{\mathcal W}}
\newcommand\XX{{\mathcal X}}
\newcommand\YY{{\mathcal Y}}
\newcommand\ZZ{{\mathcal Z}}

\title{A Torelli type theorem for exp-algebraic curves}

\author{Indranil Biswas}
\address{School of Mathematics, Tata Institute of Fundamental
Research, Homi Bhabha Road, Bombay 400005, India}

\email{indranil@math.tifr.res.in}

\author{Kingshook Biswas}
\address{RKM Vivekananda University, Belur Math, WB-711 202, India}

\email{kingshook@rkmvu.ac.in}


\begin{abstract} An exp-algebraic curve consists of a compact Riemann surface $S$
together with $n$ equivalence classes of germs of meromorphic functions modulo
germs of holomorphic functions, $\HH = \{ [h_1], \cdots, [h_n] \}$, with poles of
orders $d_1, \cdots, d_n \geq 1$ at points $p_1, \cdots, p_n$. This data determines
a space of functions $\OO_{\HH}$ (respectively, a space of $1$-forms $\Omega^0_{\HH}$)
holomorphic on the punctured surface $S' = S - \{p_1, \cdots, p_n\}$ with exponential
singularities at the points $p_1, \cdots, p_n$ of types $[h_1], \cdots, [h_n]$, i.e.,
near $p_i$ any $f \in \OO_{\HH}$ is of the form $f = ge^{h_i}$ for some germ of meromorphic
function $g$ (respectively, any $\omega \in \Omega^0_{\HH}$ is of the form $\omega = \alpha e^{h_i}$
for some germ of meromorphic $1$-form).

For any $\omega \in \Omega^0_{\HH}$ the completion of $S'$ with respect to
the flat metric $|\omega|$ gives a space $S^* = S' \cup \RR$ obtained by adding a finite
set $\RR$ of $\sum_i d_i$ points, and it is known that integration along curves
produces a nondegenerate pairing
of the relative homology $H_1(S^*, \RR ; \C)$ with the de Rham cohomology group defined by
$H^1_{dR}(S, \HH) := \Omega^0_{\HH}/d\OO_{\HH}$.

There is a degree zero line bundle $L_{\HH}$ associated to an exp-algebraic curve, with a
natural isomorphism between $\Omega^0_{\HH}$ and the space $W_{\HH}$ of meromorphic
$L_{\HH}$-valued $1$-forms which are holomorphic on $S'$, so that $H_1(S^*, \RR ; \C)$ maps
to a subspace $K_{\HH} \subset W^*_{\HH}$.
We show that the exp-algebraic curve $(S, \HH)$ is determined
uniquely by the pair $(L_{\HH},\, K_{\HH} \subset W^*_{\HH})$.
\end{abstract}

\bigskip

\maketitle

\tableofcontents

\section{Introduction}

\medskip

\medskip

A choice of nonconstant meromorphic function $z$ on a compact Riemann surface $S$ realizes $S$ as a finite sheeted branched
covering of the Riemann sphere $\widehat{\C}$. {\it Log-Riemann surfaces of finite type}
are certain branched coverings, in a generalized sense, of $\C$ by a punctured compact Riemann surface,
namely, which are given by certain transcendental functions of
infinite degree. Formally a log-Riemann surface consists of a Riemann surface together with a local holomorphic
diffeomorphism $\pi$ from the surface to $\C$ such that the set of points $\RR$ added to the surface,
in the completion $S^* = S' \cup \RR$ with respect to the path-metric
induced by the flat metric $|d\pi|$, is discrete. Log-Riemann surfaces were
defined and studied in \cite{bpm1} (see also \cite{bpm4}), where it was shown that
the map $\pi$ restricted to any small enough punctured
metric neighbourhood of a point $w^*$ in $\RR$ gives a covering of a punctured disc in $\C$,
and is thus equivalent to
either $(z \longmapsto z^n)$ restricted to a punctured disc $\{ 0 < |z| < \epsilon \}$
(in which case we say $w^*$ is a ramification point of order $n$) or to $(z \mapsto e^z)$
restricted to a half-plane $\{ \Re z < C \}$
(in which case we say $w^*$ is a ramification point of infinite order).

\medskip

A log-Riemann surface is said to be of finite type if it has finitely many ramification points
and finitely generated fundamental group. We will only consider those for which the set of
infinite order ramification points is nonempty (otherwise the map $\pi$ has finite degree
and is given by a meromorphic function on a compact Riemann surface). In \cite{bpm2}, \cite{bpm3},
uniformization theorems were proved for log-Riemann surfaces of finite type, which imply that a
log-Riemann surface of finite type is given by a pair $(S' = S - \{p_1, \cdots, p_n \}, \pi)$, where
$S$ is a compact Riemann surface, and $\pi$ is a meromorphic function on the punctured surface $S'$
such that the differential $d\pi$ has essential
singularities at the punctures of a specific type, namely {\it exponential singularities}.

\medskip

Given a germ of meromorphic function $h$ at a point $p$ of a Riemann surface, a function $f$ with an isolated singularity
at $p$ is said to have an exponential singularity of type $h$ at $p$ if locally $f = g e^h$ for
some germ of meromorphic function $g$ at $p$, while a 1-form $\omega$ is said to have an
exponential singularity of type $h$ at $p$ if locally $\omega = \alpha e^h$ for
some germ of meromorphic 1-form $\alpha$ at $p$. Note that the spaces of germs of functions and 1-forms
with exponential singularity of type $h$ at $p$ only depend on the equivalence class $[h]$ in the space
$\mathcal{M}_p/\OO_p$ of germs of meromorphic functions at $p$ modulo germs of holomorphic functions at $p$.

\medskip

Thus the uniformization theorems of \cite{bpm2}, \cite{bpm3} give us
$n$ germs of meromorphic functions $h_1, \cdots, h_n$ at the punctures $p_1,\cdots,p_n$, with poles of orders
$d_1, \cdots, d_n \geq 1$ say, such that near a puncture $p_j$ the map $\pi$ is of the form $\int g_j e^{h_j} dz$,
where $g_j$ is a germ of meromorphic function near $p_j$ and $z$ a local coordinate near $p_j$. The punctures correspond
to ends of the log-Riemann surface, where at each puncture $p_j$, $d_j$ infinite order ramification points are added
in the metric completion, so that the total number of infinite order ramification points is $\sum_j d_j$. The $d_j$
infinite order ramification points added at a puncture $p_j$ correspond to the $d_j$ directions of approach to the
puncture along which $\Re h_j \to -\infty$ so that $e^{h_j}$ decays exponentially and $\int^z g_j e^{h_j} dz$ converges.
In the case of genus zero with one
puncture for example, which is considered in \cite{bpm2}, $\pi$ must have the form $\int R(z) e^{P(z)} dz$
where $R$ is a rational function and $P$ is a polynomial of degree equal to the number of infinite
order ramification points.

\medskip

In \cite{kb1}, certain spaces of functions and $1$-forms on a log-Riemann surface $S^*$ of finite type
were defined, giving rise to a de Rham cohomology group $H^1_{dR}(S^*)$. The integrals of the $1$-forms considered
along curves in $S^*$ joining the infinite ramification points converge, giving rise to a pairing between $H^1_{dR}(S^*)$
and $H_1(S^*, \RR ; \C)$, which was shown to be nondegenerate (\cite{kb1}).

\medskip

The spaces of functions and $1$-forms defined were observed to depend only on the types
$h_1, \cdots, h_n$ of the exponential singularities, and so a notion less rigid than that of a log-Riemann surface
was defined, namely the notion of an {\it exp-algebraic curve}, which consists of a compact Riemann surface $S$
together with $n$ equivalence classes of germs of meromorphic functions modulo
germs of holomorphic functions, $\HH = \{ [h_1], \cdots, [h_n] \}$, with poles of
orders $d_1, \cdots, d_n \geq 1$ at points $p_1, \cdots, p_n$. The relevant spaces of
 functions and $1$-forms with exponential singularities at $p_1, \cdots, p_n$ of types $[h_1], \cdots, [h_n]$
 can then be defined as follows:
\begin{align*}
\MM_{\HH}\,:= & \{ f \,\mid\, f \hbox{ meromorphic function on } S' \hbox{ with exponential singularities}\\
& \hbox{   of types } [h_1], \cdots, [h_n] \} \\
\OO_{\HH}\,:= & \{ f \in \MM_{\HH} \,\mid\, f \hbox{ holomorphic on } S' \} \\
\Omega_{\HH}\,:= & \{ \omega \,\mid\, \omega \hbox{ meromorphic 1-form on } S' \hbox{ with exponential singularities}\\
& \hbox{   of types } [h_1], \cdots, [h_n] \} \\
\Omega^0_{\HH}\,:= & \{ \omega \in \Omega_{\HH} \,\mid\, \omega \hbox{ holomorphic on } S' \}.
\end{align*}

\medskip

For $f \,\in\, \MM_{\HH}$ (respectively, $\omega \,\in\, \Omega_{\HH}$) we can define a
divisor $(f) = \sum_{p \in S} n_p \cdot p$
(respectively, $(\omega) = \sum_{p \in S} m_p \cdot p$) by $n_p = ord_p(f)$ if $p \in S'$ and
$n_p = ord_{p_i}(g)$ if $p = p_i$, where $g$ is a germ of meromorphic function at $p_i$ such that $f = ge^{h_i}$
(respectively, $m_p = ord_p(\omega)$ if $p \in S'$ and
$n_p = ord_{p_i}(\alpha)$ if $p = p_i$, where $\alpha$ is a germ of meromorphic $1$-form
at $p_i$ such that $\omega = \alpha e^{h_i}$).

\medskip

In \cite{kb1} it is shown how to naturally associate to an exp-algebraic curve $(S,\, {\mathcal{H}})$
a degree zero line bundle $L_{\mathcal{H}}$ together with
a meromorphic connection $\nabla_{\mathcal{H}}$ with poles at $p_1, \cdots, p_n$. The connection
$1$-form of $\nabla_{\HH}$ near $p_i$ is given (with respect to an appropriate local trivialization)
by $dh_i$, so that the pair $(L_{\HH}, \nabla_{\HH})$ determines the exp-algebraic curve $(S,\, \HH)$.
There are naturally defined isomorphisms between the space of meromorphic sections of
$L_{\HH}$ (respectively, the
space of meromorphic $L_{\HH}$-valued $1$-forms) and $\MM_{\HH}$ (respectively, $\Omega_{\HH}$),
such that a meromorphic section $s$ of $L_{\HH}$ (respectively, a meromorphic $L_{\HH}$-valued $1$-form $\alpha$)
maps to an $f \in \MM_{\HH}$ with the same divisor as $s$ (respectively, an $\omega \in \Omega_{\HH}$
with the same divisor as $\alpha$).

\medskip

In particular the space $W_{\HH}$ of meromorphic $L_{\HH}$-valued $1$-forms which
are holomorphic on $S'$ is naturally isomorphic to the space $\Omega^0_{\HH}$. Fixing an $f \in \OO_{\HH}$
inducing a log-Riemann surface structure on $S$, with completion $S^* = S' \cup \RR$, the $1$-forms in
$\Omega^0_{\HH}$ can be integrated along curves in $H_1(S^*, \RR ; \C)$, giving a map
$$H_1(S^*, \RR ; \C) \,\longrightarrow\, (\Omega^0_{\HH})^* \,\simeq\, W^*_{\HH}\, .$$
Let $K_{\HH} \subset W^*_{\HH}$ denote
the image of $H_1(S^*, \RR ; \C)$ in $W^*_{\HH}$. Then our Torelli-type theorem for exp-algebraic
curves states that the pair $(L_{\HH},\, K_{\HH})$ determines the exp-algebraic curve $(S, \,\HH)$:

\medskip

\begin{theorem} \label{mainthm} Let $(S,\, \HH_1), (S,\, \HH_2)$ be two exp-algebraic curves with the
same underlying Riemann surface $S$, and the same set of punctures $p_1, \cdots, p_n$. Suppose that $H_1(S^*_1, \RR ; \C)$
is nontrivial, that the line bundles $L_{\HH_1}, L_{\HH_2}$ are isomorphic and that
the induced isomorphism $W^*_{\HH_1} \,\longrightarrow\, W^*_{\HH_2}$
maps $K_{\HH_1}$ to $K_{\HH_2}$. Then $\HH_1 \,=\, \HH_2$.
\end{theorem}

\medskip

Finally, we mention briefly some appearances of functions with exponential singularities in the literature.
Certain functions with exponential singularities, namely the $n$-point {\it Baker-Akhiezer functions}
(\cite{baker}, \cite{akhiezer}), have been used in the algebro-geometric integration of
integrable systems (see, for example, \cite{krichever3},
\cite{krichever2} and the surveys \cite{krichever1}, \cite{dubrovin},
\cite{krichnov}, \cite{dubkrichnov}). Given a divisor $D$
on $S'$, an $n$-point Baker-Akhiezer function (with respect to the data $(\{p_j\}, \{h_j\}, D)$)
is a function $f$ in the space $\MM(S^*)$ satisfying the additional properties that the divisor $(f)$ of
zeroes and poles of $f$ on $S'$ satisfies $(f) + D \geq 0$, and that $f \cdot e^{-h_j}$ is holomorphic at $p_j$ for all
$j$. For $D$ a non-special divisor of degree at least $g$,
the space of such Baker-Akhiezer functions is known to have dimension $deg D - g + 1$.

\medskip

Functions and differentials with exponential singularities
on compact Riemann surfaces have also been studied by Cutillas Ripoll
(\cite{cutillas1}, \cite{cutillas2}, \cite{cutillas3}), where they
arise naturally in the solution of the {\it Weierstrass problem} of realizing arbitrary divisors
on compact Riemann surfaces, and by Taniguchi (\cite{taniguchi1}, \cite{taniguchi2}), where entire functions satisfying certain
topological conditions (called ``structural finiteness") are shown to be precisely those entire functions
whose derivatives have an exponential singularity at ${\infty}$, namely functions of the
form $\int Q(z) e^{P(z)} \ dz$, where $P, Q$ are polynomials.

\medskip

{\bf Acknowledgements.} This work grew out of a visit of the second author to TIFR, Mumbai.
He would like to thank TIFR for its hospitality.

\bigskip

\section{Log-Riemann surfaces of finite type and exp-algebraic curves}

\medskip

We recall some basic definitions and facts from \cite{bpm1}, \cite{bpm2},
\cite{bpm3}.

\medskip

\begin{definition} A log-Riemann surface is a pair $(S, \pi)$,
where $S$ is a Riemann surface and $\pi \,:\, S \,\longrightarrow\, \C$ is
a local holomorphic diffeomorphism such that the set of points $\RR$ added
to $S$ in the completion $S^* \,:=\, S \sqcup \RR$
with respect to the path metric induced by the flat
metric $|d\pi|$ is discrete.
\end{definition}

\medskip

The map $\pi$ extends to the metric completion $S^*$ as
a $1$-Lipschitz map. In \cite{bpm1} it is shown that the map $\pi$
restricted to a sufficiently small punctured metric neighbourhood
$B(w^*, r) - \{w^*\}$ of a ramification point is a covering of
a punctured disc $B(\pi(w^*), r) - \{\pi(w^*)\}$ in $\C$, and
so has a well-defined degree $1 \leq n \leq +\infty$, called the order
of the ramification point (we assume that the order is always at least $2$,
since order one points can always be added to $S$ and $\pi$ extended
to these points in order to obtain a log-Riemann surface).

\medskip

\begin{definition} A log-Riemann surface is of finite type if it has
finitely many ramification points and finitely generated fundamental group.
\end{definition}

\medskip

For example, the log-Riemann surface given by $(\C, \pi = e^z)$ is of finite type (with the
metric $|d\pi|$ it is isometric to the Riemann surface of the logarithm, which is simply
connected, with a single ramification point of infinite order), as is the log-Riemann surface
given by the Gaussian integral $(\C, \pi = \int e^{z^2} dz)$, which has two ramification points,
both of infinite order, as in the figure below:

\medskip

{\hfill {\centerline {\psfig {figure=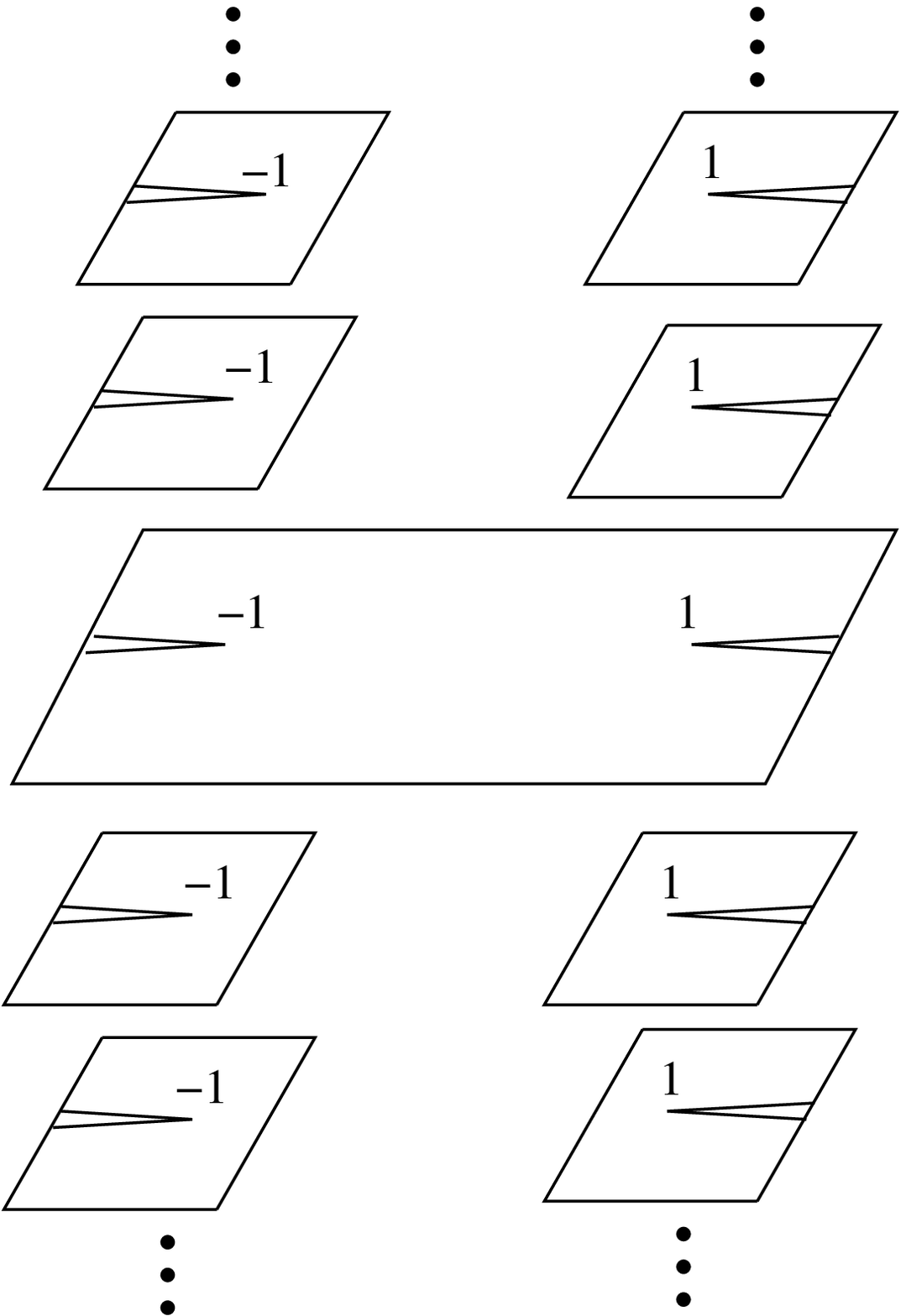,height=7cm}}}}

\medskip

\centerline{Log-Riemann surface of the Gaussian integral}

\medskip

In \cite{bpm3}, it is shown that a log-Riemann surface of finite type (which has at least one
infinite order ramification point) is of the form
$(S', \pi)$, where $S'$ is a punctured compact Riemann surface $S' = S - \{p_1, \cdots, p_n\}$
and $\pi$ is meromorphic on $S'$ and $d\pi$ has exponential singularities at the punctures
$p_1, \cdots, p_n$. Let $h_1, \cdots, h_n$ be the types of the exponential singularities
of $d\pi$ at the punctures $p_1, \cdots, p_n$.
As described in \cite{bpm3}, each puncture $p_j$ corresponds to an
end of the log-Riemann surface where $d_j$ infinite order ramification points are added,
$d_j$ being the order of the pole of $h_j$ at $p_j$.

\medskip

Let $w^*$ be an infinite order ramification point associated to a puncture $p_j$.
An $\epsilon$-ball $B_{\epsilon}$ around $w^*$ is isometric to the $\epsilon$-ball
around the infinite order ramification
point of the Riemann surface of the logarithm (given by cutting and pasting infinitely many
discs together), and there is an argument function $$\arg_{w^*} \,:\, B^*_{\epsilon}
\,\longrightarrow\, \R$$ defined
on the punctured ball $B^*_{\epsilon}$. While the function $\pi$, which is of the form
$\pi = \int e^{h_j} \alpha_j$ in a punctured neighbourhood of $p_j$ for some meromorphic $1$-form
$\alpha_j$, extends continuously to $w^*$
for the metric topology on $S^*$, in general functions of the form $f = \int e^{h_j} \alpha$
(where $\alpha$ is a $1$-form meromorphic near $p_j$) do {\bf not} extend continuously
to $w^*$ for the metric topology (\cite{bpm4}). Limits of these functions in sectors
$\{ p \in B^*_{\epsilon} \,\mid\, c_1 < \arg_{w^*}(p) < c_2 \}$ do exist however and are independent
of the sector; we say that the function is {\it Stolz continuous} at points of $\RR$.

\medskip

\begin{definition} Define spaces of functions and $1$-forms on $S^*$:
\begin{align*}
\MM(S^*) & := \{ f \hbox{ meromorphic function on } S'\,\mid\, f \hbox{ has exponential singularities at } \\
         & \quad \quad p_1, \cdots, p_n \hbox{ of types } h_1, \cdots, h_n \} \\
\OO(S^*) & := \{ f \in \MM(S^*) \,\mid\, f \hbox{ holomorphic on } S' \} \\
\Omega(S^*) & := \{ \omega \hbox{ meromorphic } 1-\hbox{form on } S' \,\mid\, \omega \hbox{ has exponential singularities at } \\
         & \quad \quad p_1, \cdots, p_n \hbox{ of types } h_1, \cdots, h_n \} \\
\Omega^0(S^*) & := \{ \omega \in \Omega(S^*) \,\mid\, \omega \hbox{ holomorphic on } S' \}
\end{align*}
\end{definition}

\medskip

We remark that these are simply the spaces $\MM_{\HH}, \OO_{\HH}, \Omega_{\HH}, \Omega^0_{\HH}$
defined in the introduction, where $\HH = \{ [h_1], \cdots, [h_n] \}$.
Functions in $\MM(S^*)$ are Stolz continuous at points of $\RR$ taking the value $0$ there. The integrals
of $1$-forms $\omega$ in $\Omega_{II}(S^*)$ over curves $\gamma \,:\, [a,\, b]
\,\longrightarrow\, S^*$ joining points $w^*_1, w^*_2$ of
$\RR$ converge if $\gamma$ is disjoint from the poles of $\omega$ and
tends to these points through sectors $$\{ p \in B^*_{\epsilon} \,\mid\, c_1 < \arg_{w^*_1}(p) < c_2 \},\ \
\{ p \in B^*_{\epsilon} \,\mid\, c_1 < \arg_{w^*_2}(p) < c_2 \}$$ (since any primitive of $\omega$ on a sector
is Stolz continuous).

\medskip

The definitions of the above spaces only depend on the types $\{ [h_i] \in \MM_{p_i}/\OO_{p_i} \}$ of the exponential
singularities of the $1$-form $d\pi$, which do not change if $d\pi$ is multiplied by a meromorphic function.
It is natural to define then a structure less rigid than that of a log-Riemann surface of finite type.

\medskip

\begin{definition}[{Exp-algebraic curve}] Given a punctured compact Riemann surface $S' = S - \{p_1, \cdots, p_n\}$,
two meromorphic functions $\pi_1, \pi_2$ on $S'$ inducing log-Riemann surface structures of finite type
are considered equivalent if $d\pi_1/d\pi_2$ is meromorphic on the compact surface $S$. An exp-algebraic curve is an
equivalence class of such log-Riemann surface structures of finite type.
\end{definition}

\medskip

It follows from the uniformization theorem of \cite{bpm3} that
an exp-algebraic curve is given by the data of a punctured compact Riemann surface and
$n$ (equivalence classes of) germs of meromorphic functions $\mathcal{H} = \{ [h_i] \in \MM_{p_i}/\OO_{p_i} \}$
with poles at the punctures.

\medskip

We can associate
a topological space $\widehat{S}$ to an exp-algebraic curve, given as a set by $\widehat{S} = S' \cup \RR$,
where $\RR$ is the set of infinite ramification points added with respect to any map $\pi$ in the
equivalence class of log-Riemann surfaces of finite type, and the topology is the weakest
topology such that all maps $\widetilde{\pi}$ in the equivalence class extend continuously to $\widehat{S}$.

\medskip

Finally, for a meromorphic function $f$ on $S'$ (respectively, meromorphic $1$-form $\omega$ on $S'$)
with exponential singularities of
types $h_1, \cdots, h_n$ at points $p_1, \cdots, p_n$ we can define a divisor $(f) = \sum_{p \in S} n_p \cdot p$
(respectively, $(\omega) = \sum_{p \in S} m_p \cdot p$) by $n_p = ord_p(f)$ if $p \in S'$ and
$n_p = ord_{p_i}(g)$ if $p = p_i$, where $g$ is a germ of meromorphic function at $p_i$ such that $f = ge^{h_i}$
(respectively, $m_p = ord_p(\omega)$ if $p \in S'$ and
$n_p = ord_{p_i}(\alpha)$ if $p = p_i$, where $\alpha$ is a germ of meromorphic $1$-form
at $p_i$ such that $\omega = \alpha e^{h_i}$).

\medskip

Note that the divisor $(f)$ can also be defined by $n_p = Res(df/f, p)$, so it follows from the Residue
Theorem applied to the meromorphic $1$-form $df/f$ that the divisor $(f)$ has degree zero.

\medskip

\section{Exp-algebraic curves and line bundles with meromorphic connections}

\medskip

Let $(S, \mathcal{H} \,=\, \{h_1,\, \cdots,\, h_n\})$ be an exp-algebraic curve, where $S$ is a compact Riemann surface of genus $g$
and $h_1, \cdots, h_n$ are germs of meromorphic functions at points $p_1, \cdots, p_n$. Let $\Omega(S)$ be the
space of holomorphic $1$-forms on $S$. The data $\mathcal{H}$ defines a degree zero line bundle $L_{\mathcal{H}}$
together with a transcendental section $s_{\mathcal H}$ of this line bundle which is non-zero on the punctured
surface $S'$ as follows:

\medskip

Solving the Mittag-Leffler problem locally for the distribution $\{h_1, \cdots, h_n\}$ gives meromorphic
functions on an open cover such that the differences are holomorphic on intersections, and hence gives
an element of $H^1(S, \OO)$. Under the exponential this gives a degree zero line bundle as an element
of $H^1(S, \OO^*)$. Explicitly this is constructed as follows:

\medskip

Let $$B_1, \,\cdots,\, B_n$$ be pairwise disjoint coordinate disks around the
punctures $p_1, \cdots, p_n$ and let $V$ be an open subset of $S'$ intersecting each disk $B_i$ in
an annulus $U_i = V \cap B_i$ around $p_i$ such that $\{B_1, \cdots, B_n, V \}$ is an open cover of $S$.
Define a line bundle $L_{\mathcal{H}}$ by taking the functions $e^{-h_i}$ to be the transition functions
for the line bundle on the intersections $U_i$. Define a holomorphic non-vanishing section
of $L_{\mathcal{H}}$ on $S'$ by:

\begin{align*}
s_{\mathcal{H}} := & \quad \quad 1 & \quad \hbox{ on } V \\
                   & \quad \quad e^{-h_i} & \quad \hbox{ on } B_i - \{p_i\}
\end{align*}

Define a connection $\nabla_{\mathcal{H}}$ on $L_{\mathcal{H}}$ by declaring
that $\nabla_{\mathcal{H}}(s_{\mathcal{H}}) \,=\, 0$. Then for any holomorphic section
$s$ on $V$, we have $s = f s_{\mathcal{H}}$ for some holomorphic function $f$, and also
$\nabla_{\mathcal{H}}(s) \,=\, df s_{\mathcal{H}}$, so $\nabla_{\mathcal{H}}$ is holomorphic
on $V$. On each disk $B_i$, letting $s_i$ be the section which is a constant equal to $1$
on $B_i$ (with respect to the trivialization on $B_i$), for any holomorphic section
$s$ on $B_i$, we have $s = f s_i$ for some holomorphic function $f$, and
also $s_i \,=\, e^{h_i} s_{\mathcal{H}}$,
so
\begin{align*}
\nabla_{\mathcal{H}}(s) & = \nabla_{\mathcal{H}}(f s_i) \\
                        & = \nabla_{\mathcal{H}}(fe^{h_i}s_{\mathcal{H}}) \\
                        & = (df + f dh_i)e^{h_i} s_{\mathcal{H}} \\
                        & = (df + f dh_i) s_i \\
\end{align*}
thus the connection $1$-form of $\nabla_{\mathcal{H}}$ with respect to $s_i$ is given by
$dh_i$, so $\nabla_{\mathcal{H}}$ is meromorphic on $B_i$ with a single pole at $p_i$
of order $d_i + 1 \geq 2$.

\medskip

Let $s^*_{\HH}$ be the unique section of the dual bundle $L^*_{\mathcal{H}}$ on $S'$
such that $s_{\HH} \otimes s^*_{\HH}\,=\, 1$ on $S'$. Then for any non-zero meromorphic
section
$s$ of $L_{\HH}$, the function $$f \,:= \,s \otimes s^*_{\mathcal{H}}$$ is meromorphic on $S'$
with exponential singularities at $p_1, \cdots, p_n$ of types $h_1, \cdots, h_n$, and the
divisors of $s$ and $f$ coincide. Thus the
line bundle $L_{\mathcal{H}}$ has degree zero. In summary we have:

\medskip

\begin{theorem} \label{existence} For any log-Riemann surface of finite type $S^*$, the line bundle $L_{\mathcal{H}}$
has degree zero and the maps
$$s \,\longmapsto \, s \otimes s^*_{\HH},\ \ f \,\longmapsto \,f \cdot s_{\HH}$$
(respectively, $$\alpha \,\longmapsto\, \alpha \otimes s^*_{\HH},
\ \ \omega \,\longmapsto\, \omega \cdot s_{\HH})$$
are mutually inverse isomorphisms between the spaces of meromorphic sections of
$L_{\HH}$ and $\MM(S^*)$ (respectively, the spaces of meromorphic $L_{\HH}$-valued
$1$-forms and $\Omega(S^*)$) preserving divisors.

\medskip

In particular the vector spaces $\MM(S^*), \OO(S^*), \Omega(S^*), \Omega_{II}(S^*), \Omega^0(S^*)$
are non-zero.
\end{theorem}

\medskip

\begin{proof}
Since the isomorphisms above preserve divisors,
the spaces $\OO(S^*), \Omega^0(S^*)$ correspond to the spaces of meromorphic sections of $L_{\HH}$ and
meromorphic $L_{\HH}$-valued $1$-forms which are holomorphic on $S'$, both of which are non-empty.
\end{proof}

\medskip

\begin{prop} The correspondence $\HH \mapsto (L_{\HH}, \nabla_{\HH})$ gives a one-to-one
correspondence between exp-algebraic structures on $S$ and degree zero line bundles on $S$ with
meromorphic connections with all poles of order at least two, zero residues, and trivial monodromy.
\end{prop}

\medskip

\begin{proof}
Since the connection $1$-form of $\nabla_{\HH}$ is given by $dh_i$ on $B_i$,
all residues of $\nabla_{\HH}$ are equal to zero, while the monodromy of $\nabla_{\HH}$ is trivial
since $s_{\HH}$ is a single-valued horizontal section.

\medskip

Conversely, given such a meromorphic connection $\nabla$ on a degree zero line bundle $L$, if $p_1, \cdots, p_n$
are the poles of $\nabla$ and $\omega_1, \cdots, \omega_n$ are the connection $1$-forms of $\nabla$ with
respect to trivializations near $p_1, \cdots, p_n$, then each $\omega_i$ has zero residue at $p_i$ and pole
order at least two, hence there exist meromorphic germs $h_1, \cdots, h_n$ near $p_1, \cdots, p_n$ such that
$\omega_i = dh_i$. We obtain an exp-algebraic curve $(S, \HH(L, \nabla))$.

\medskip

It is clear for an exp-algebraic curve $(S, \HH)$ that $\HH(L_{\HH}, \nabla_{\HH}) = \HH$, so the
correspondences are inverses of each other.
\end{proof}

\medskip

Finally we remark that by Serre Duality, the degree zero line bundle $L_{\mathcal{H}}$, given as an element of $H^1(S, \OO)$,
can also be described as an element of $H^0(S, \Omega)^* = \Omega(S)^*$ using residues, as the linear
functional
\begin{align*}
Res_{\mathcal{H}} : \Omega(S) & \rightarrow \C \\
                      \xi     & \mapsto  \sum_i Res(\xi \cdot h_i, p_i) \\
\end{align*}

\medskip

\section{Torelli-type theorem for exp-algebraic curves}

\medskip

We proceed to the proof of Theorem \ref{mainthm}. We will need the following theorems from
\cite{kb1} and \cite{gusman1}:

\medskip

\begin{theorem}[\cite{kb1}]\label{pairing}
 Let $S^* = S' \cup \RR$ be a log-Riemann surface of finite type,
and let $H^1_{dR,0}(S^*) \,=\, \Omega^0(S^*)/d\OO(S^*)$. Then the pairing
$H_1(S^*, \RR ; \C) \times H^1_{dR,0}(S^*) \,\longrightarrow\, \C$,
given by integration along curves, is nondegenerate.
\end{theorem}

\medskip

\begin{theorem}[{Gusman, \cite{gusman1}}] \label{approxn}
Let $S$ be a compact Riemann surface and $E \subset S$ a closed
subset such that $S - E$ has finitely many connected components $V_1, \cdots, V_m$,
and for each $i$ let $q_i$ be a point of $V_i$. Then any continuous function $f$ on $E$
which is holomorphic in the interior of $E$ can be uniformly approximated on $E$ by functions
meromorphic on $S$ with poles only in the set $\{q_1, \cdots, q_m\}$.
\end{theorem}

\medskip

Let $(S, \HH_1), (S, \HH_2)$ be two exp-algebraic curves with the same underlying Riemann surface $S$
and the same set of punctures $p_1, \cdots, p_n$,
and suppose the hypothesis of Theorem 1.1 holds, namely the line bundles $L_{\HH_1}, L_{\HH_2}$
are isomorphic and the induced isomorphism $$W^*_{\HH_1}
\,\longrightarrow\, W^*_{\HH_2}$$ maps $K_{\HH_1}$ to $K_{\HH_2}$.
Since the spaces $\MM_{\HH_1}, \MM_{\HH_2}$ are isomorphic to the spaces of meromorphic sections of
$L_{\HH_1}$ and $L_{\HH_2}$ respectively, there is an induced isomorphism $\MM_{\HH_1}
\,\longrightarrow\, \MM_{\HH_2}$
which preserves divisors. We fix non-zero functions $f_i \,\in\, \MM_{\HH_i},\  i = 1,2$ which correspond
to each other under this isomorphism, and let $S^*_i = S' \cup \RR_i$ denote the completions
induced by the corresponding log-Riemann surface structures. We also fix a meromorphic
$1$-form $\alpha_0$ on $S$. Then the divisor preserving
isomorphisms $\MM_{\HH_1} \,\longrightarrow\, \MM_{\HH_2}$ and $\Omega_{\HH_1}
\,\longrightarrow\,\Omega_{\HH_2}$ induced by the
isomorphism $L_{\HH_1} \,\longrightarrow\, L_{\HH_2}$ can be expressed as
$$g \cdot f_1 \,\longmapsto\, g \cdot f_2
\ \  \text{ and } \ \
g f_1 \cdot \alpha_0 \longmapsto g f_2 \cdot \alpha_0$$ respectively, where $g$
varies over all meromorphic functions on $S$.

\medskip

\begin{lemma} \label{exact} For any meromorphic function $g_2$ on $S$ such that $g_2 f_2 \in \OO_{\HH_2}$,
the $1$-form $g_2 f_1\left(\frac{df_1}{f_1} - \frac{df_2}{f_2}\right)$ lies in the space $d\OO_{\HH_1}$.
\end{lemma}

\medskip

\begin{proof}
The hypothesis of Theorem \ref{mainthm} implies that for any $\gamma_1 \in H_1(S^*_1, \RR_1 ; \C)$,
there is a $$\gamma_2 \,\in\, H_1(S^*_2, \RR_2 ; \C)$$ such that $\int_{\gamma_1} gf_1 \alpha_0 = \int_{\gamma_2} gf_2 \alpha_0$
for all meromorphic functions $g$ on $S$ such that $g f_i \alpha_0$ is holomorphic on $S'$ for $i = 1,2$.
If $g_2$ is a meromorphic function on $S$ such that $g_2 f_2$ is holomorphic on $S'$, then
$$d(g_2 f_2) \,=\, g f_2 \alpha_0$$
for some meromorphic function $g$ on $S$ such that $g f_2 \alpha_0$ is holomorphic on $S'$. Since the isomorphism
$\Omega(\HH_1) \,\longrightarrow\, \Omega(\HH_2)$ is divisor preserving, we have that $g f_1 \alpha_0$ is also holomorphic on $S'$, so
for any $\gamma_1 \,\in\, H_1(S^*_1, \RR_1 ; \C)$ there is a $\gamma_2 \in H_1(S^*_2, \RR_2 ; \C)$ such that
$$
\int_{\gamma_1} gf_1 \alpha_0 = \int_{\gamma_2} gf_2 \alpha_0 = \int_{\gamma_2} d(g_2 f_2) = 0
$$
Since this is true for all $\gamma_1 \in H_1(S^*_1, \RR_1 ; \C)$, it follows from Theorem \ref{pairing}
that $gf_1 \alpha_0 \in d\OO_{\HH_1}$, so there exists a meromorphic function $g_1$ on $S$ such that $g_1 f_1$
is holomorphic on $S'$ and $g f_1 \alpha_0 = d(g_1 f_1)$.

\medskip

It follows from the equalities $g f_i \alpha_0 = d(g_i f_i), i = 1,2$ that
$$
dg_1 + g_1 \frac{df_1}{f_1} = g \alpha_0 = dg_2 + g_2 \frac{df_2}{f_2},
$$
hence
$$
dg_1 + g_1 \frac{df_1}{f_1} = \left(dg_2 + g_2 \frac{df_1}{f_1}\right) + g_2 \left(\frac{df_2}{f_2} - \frac{df_1}{f_1}\right),
$$
so multiplying above by $f_1$ gives
$$
g_2 f_1 \left(\frac{df_2}{f_2} - \frac{df_1}{f_1}\right) = d(g_1 f_1) - d(g_2 f_1) \in d\OO_{\HH_1}
$$
as required, since $g_1 f_1, g_2 f_1 \in \OO_{\HH_1}$ (note that $g_2 f_2 \in \OO_{\HH_2}$ implies $g_2 f_1 \in \OO_{\HH_1}$).
\end{proof}

\medskip

\begin{proof}[Proof of Theorem \ref{mainthm}] We consider different cases:

\medskip

\noindent Case 1. The genus of $S$ is at least one:

\medskip

In this case there exists a closed curve $\gamma$ disjoint from the punctures $p_1, \cdots, p_n$ and the poles and
zeroes of $f_1, f_2$ such that $S - \gamma$ is connected. Fix a non-zero meromorphic function $g_0$ on $S$ such that
$g_0 f_1$ (and hence also $g_0 f_2$) is holomorphic on $S'$.

\medskip

If the meromorphic $1$-form $\omega = \frac{df_1}{f_1} - \frac{df_2}{f_2}$ (which is holomorphic outside
the punctures and the zeroes and poles of $f_1, f_2$)
is not identically zero, then we can choose a continuous function $u$ on $\gamma$ such that $\int_{\gamma} u g_0 f_1 \omega \neq 0$
(since the $1$-form $g_0 f_1 \omega$ is holomorphic and not identically zero on $\gamma$). By Theorem \ref{approxn}, since
$S - \gamma$ is connected and contains $p_1$, we can choose a meromorphic function $v$ on $S$ which is holomorphic
on $S - \{p_1\}$ and uniformly close enough to $u$ on $\gamma$ such that $\int_{\gamma} v g_0 f_1 \omega \neq 0$.
Letting $g_2 = v g_0$, we have that $g_2 f_2$ is holomorphic on $S'$ and $\int_{\gamma} g_2 f_1 \omega \neq 0$,
contradicting Lemma \ref{exact}. It follows that $df_1/f_1 \equiv df_2/f_2$, from which it follows easily
that $\HH_1 = \HH_2$.

\medskip

\noindent Case 2. The genus of $S$ is zero and the number $n$ of punctures is at least two:

\medskip

In this case $S = \widehat{\C}$ and we may assume $p_1 = 0, p_2 = \infty$. Fix a non-zero polynomial $P$ such that
$Pf_1, Pf_2$ are holomorphic on $S'$. Then by Lemma \ref{exact}, for all $k \in \mathbb{Z}$, taking
$g_2 = z^k P(z)$ we have
$$
Res((z^k P)f_1 \omega, 0) = 0
$$
from which it follows that the Laurent series of $P f_1 \omega$ around $z = 0$ vanishes identically, hence
$\omega \equiv 0$ and $\HH_1 = \HH_2$.

\medskip

\noindent Case 3. The genus of $S$ is zero and there is only one puncture:

\medskip

In this case $S = \widehat{\C}$ and we may assume the single puncture $p_1 = \infty$, and that the
functions $f_1, f_2$ are of the form $f_i = e^{P_i}$ for some polynomials $P_1, P_2$. In this
case it follows from the main theorem of \cite{kb1} that
the dimension of $H_1(S^*_i, \RR_i ; \C)$ equals $deg(P_i) - 1$. Since $K_{\HH_1}$ and
$K_{\HH_2}$ are isomorphic by hypothesis, it follows that $deg(P_1) = deg(P_2) = d$ say, where
$d \geq 2$ since $H_1(S^*_1, \RR_1 ; \C)$ is non-trivial.

\medskip

Let $P_1(z) = a_d z^d + \ldots + a_0, P_2(z) = b_d z^d + \ldots + b_0$. Let $\gamma_1, \cdots, \gamma_{d-1}$
be a basis for $H_1(S^*_1, \RR_1 ; \C)$ as described in section 4 of \cite{kb1}, each $\gamma_k$
being a curve joining a pair of ramification points $w^*_0, w^*_k$, where $\RR_1 = \{ w^*_0, \cdots, w^*_{d-1} \}$.
By hypothesis, for each curve $\gamma_k \in H_1(S^*_1, \RR_1 ; \C)$ there is a $\gamma'_k \in H_1(S^*_2, \RR_2 ; \C)$
such that
$$
\int_{\gamma_k} Q(z) e^{P_1(z)} dz = \int_{\gamma'_k} Q(z) e^{P_2(z)} dz
$$
for all polynomials $Q$. Consider the $(d-1) \times d$ matrix
\begin{align*}
M & = \left (
\begin{matrix} \int_{\gamma_1} e^{P_1(z)} dz & \cdots &\int_{\gamma_1}
z^{d-2} e^{P_1(z)} dz & \int_{\gamma_1} z^{d-1} e^{P_1(z)} dz \\ \vdots &\vdots &\ddots & \vdots \\
\int_{\gamma_{d-1}} e^{P_1(z)} dz & \cdots &\int_{\gamma_{d-1}}
z^{d-2}e^{P_1(z)} dz & \int_{\gamma_{d-1}} z^{d-1} e^{P_1(z)} dz \\ \end{matrix} \right ) \\
  & = \left (
\begin{matrix} \int_{\gamma'_1} e^{P_2(z)} dz & \cdots &\int_{\gamma'_1}
z^{d-2} e^{P_2(z)} dz & \int_{\gamma'_1} z^{d-1} e^{P_2(z)} dz \\ \vdots &\vdots &\ddots & \vdots \\
\int_{\gamma'_{d-1}} e^{P_2(z)} dz & \cdots &\int_{\gamma'_{d-1}}
z^{d-2}e^{P_2(z)} dz & \int_{\gamma'_{d-1}} z^{d-1} e^{P_2(z)} dz \\ \end{matrix} \right ) \\
\end{align*}
It follows from Theorem III.1.5.1 of \cite{bpm4} that the $(d-1)$ $1$-forms
$$z^k e^{P_1} dz,\ \ k \,=\, 0, \cdots, (d-2)$$ span $H^1_{dR, 0}(S^*_1)$, and hence form a basis for $H^1_{dR, 0}(S^*_1)$.
Since by Theorem \ref{pairing} the pairing with $H_1(S^*_1, \RR_1 ; \C)$ is nondegenerate, it follows that
the $(d-1) \times (d-1)$ submatrix
formed by the first $(d-1)$ columns of $M$ is nonsingular, thus $M$ has rank $(d-1)$. On the other
hand, since $d(e^{P_i}) = P'_i e^{P_i} dz, i= 1,2$, it follows that
$$
M \cdot \left(\begin{matrix} a_1 \\ \vdots \\ (d-1)a_{d-1} \\ da_d \end{matrix} \right) =
M \cdot \left(\begin{matrix} b_1 \\ \vdots \\ (d-1)b_{d-1} \\ db_d \end{matrix} \right) = 0
$$
hence there is a scalar $\lambda$ such that $kb_k = \lambda ka_k, k = 1, \cdots, d-1$, so $P'_2 = \lambda P'_1$.
It follows from Lemma \ref{exact} that for any polynomial $Q$ the $1$-form
$$Q e^{P_1}(\lambda-1)P'_1 dz \,=\, Q e^{P_1} (P'_2 - P'_1) dz$$ lies
in $d\OO_{\HH_1}$. Thus if $\lambda \neq 1$, then $Q P'_1 e^{P_1} dz \in d\OO_{\HH_1}$, hence
$$
Q' e^{P_1} dz = d(Qe^{P_1}) - Q P'_1 e^{P_1}dz \in d\OO_{\HH_1}
$$
for all polynomials $Q$. Since all $1$-forms in $\Omega^0_{\HH_1}$ are of the form $P e^{P_1} dz$
for some polynomial $P$ and any $P = Q'$ for some polynomial $Q$, it follows that
$H^1_{dR,0}(S^*_1)$ is trivial,
a contradiction. Thus $\lambda = 1$, so $P'_2 = P'_1$ and hence $\HH_1 = \HH_2$.
\end{proof}

\end{document}